    \documentclass[oneside,english]{amsart}
    \usepackage[T1]{fontenc}
    \usepackage[latin9]{inputenc}
    \usepackage{geometry}
    \usepackage{verbatim}
    \usepackage{amstext}
    \usepackage{amsthm}
    \usepackage{hyperref}
    \usepackage{mathtools}
    \usepackage{graphicx}
    \usepackage{float}
    \usepackage{algorithm}
    \usepackage[noend]{algpseudocode}
    \usepackage{lipsum}
    \usepackage{xcolor}
    
    \makeatletter
    \numberwithin{equation}{section}
    \numberwithin{figure}{section}
    \theoremstyle{plain}
    \newtheorem{thm}{\protect\theoremname}
      \theoremstyle{plain}
      \newtheorem{lem}[thm]{\protect\lemmaname}
      \theoremstyle{plain}
      \newtheorem{cor}[thm]{\protect\corname}
      \theoremstyle{remark}
      
      \theoremstyle{plain}
      \newtheorem{conj}[thm]{\protect\conjname}

    \newcommand{\C}{\mathcal{C}}
    \newcommand{\F}{\mathcal{F}}
    \DeclarePairedDelimiter{\ceil}{\lceil}{\rceil}
    \DeclarePairedDelimiter{\floor}{\lfloor}{\rfloor}

    \makeatother
    
    \algnewcommand\algorithmicto{\textbf{to}}
    \algnewcommand\To{\algorithmicto{} }
    \algnewcommand\algorithmicswitch{\textbf{switch}}
    \algnewcommand\algorithmiccase{\textbf{case}}
    \algdef{SE}[SWITCH]{Switch}{EndSwitch}[1]{\algorithmicswitch\ #1\ \algorithmicdo}{\algorithmicend\ \algorithmicswitch}%
    \algdef{SE}[CASE]{Case}{EndCase}[1]{\algorithmiccase\ #1}{\algorithmicend\ \algorithmiccase}%
    \algtext*{EndSwitch}%
    \algtext*{EndCase}%
    
    \usepackage{babel}
      \providecommand{\lemmaname}{Lemma}
      \providecommand{\corname}{Corollary}
      \providecommand{\remarkname}{Remark}
      \providecommand{\conjname}{Conjecture}
    \providecommand{\theoremname}{Theorem}

    \newcommand{\blfootnote}[1]{{\renewcommand{\thefootnote}{\roman{footnote}}\footnotetext[0]{#1}}}
    
\begin{document}
    
    \title{The Toughness of Kneser Graphs}
    
 
     \author{Davin Park${}^{1\dag}$, Anthony Ostuni${}^{1\dag}$, Nathan Hayes${}^{1}$, Amartya Banerjee${}^{1}$, Tanay Wakhare${}^{2}$, Wiseley Wong${}^{1 \ast}$,  and Sebastian Cioab\u{a}${}^{3}$}
 
    \begin{abstract}
    The $\textit{toughness}$ $t(G)$ of a graph $G$ is a measure of its connectivity that is closely related to Hamiltonicity. 
    Xiaofeng Gu, confirming a longstanding conjecture of Brouwer, recently proved the lower bound $t(G) \ge \ell / \lambda - 1$ on the toughness of any connected $\ell$-regular graph, where $ \lambda$ is the largest nontrivial absolute eigenvalue of the adjacency matrix. Brouwer had also observed that many families of graphs (in particular, those achieving equality in the Hoffman ratio bound for the independence number) have toughness exactly $\ell / \lambda$. Cioab\u{a} and Wong confirmed Brouwer's observation for several families of graphs, including Kneser graphs $K(n,2)$ and their complements, with the exception of the Petersen graph $K(5,2)$. In this paper, we extend these results and determine the toughness of Kneser graphs $K(n,k)$ when $k\in \{3,4\}$ and $n\geq 2k+1$ as well as for $k\geq 5$ and sufficiently large $n$ (in terms of $k$). In all these cases, the toughness is attained by the complement of a maximum independent set and we conjecture that this is the case for any $k\geq 5$ and $n\geq 2k+1$.  
    \end{abstract}
    
    \maketitle
    
    
    \section{Introduction}
    
    \blfootnote{${}^{\dag}$ denotes joint first authorship, with the current ordering arbitrary}
    \blfootnote{${}^{\ast}$ denotes corresponding author, \email{wwong123@umd.edu}}
    \blfootnote{$^{1}$~University of Maryland, College Park, MD 20878, USA}
    \blfootnote{$^{2}$~Massachusetts Institute of Technology, Department of EECS, Cambridge, MA 02139, USA}
    \blfootnote{$^{3}$~University of Delaware, Department of Mathematical Sciences, Newark, DE 19716-2553, USA}
    
    Let $n$ and $k$ be two natural numbers such that $n\geq 2k+1$. The vertex set $\binom{[n]}{k}$ of the Kneser graph $K(n,k)$ consists of the $k$-subsets of $[n]=\{1,2,\ldots,n\}$. Two vertices $A$ and $B$ are adjacent if and only if $A\cap B = \emptyset$. The study of Kneser graphs is intertwined with the combinatorial study of intersecting sets. An \textit{intersecting family} $\F\subseteq \binom{[n]}{k}$ satisfies $F_1\cap F_2\ne \emptyset$ for any $F_1,F_2\in \F$ and corresponds to an independent set in $K(n,k)$.  The famous Erd\"os-Ko-Rado theorem \cite{erdos} states that the independence number of $K(n,k)$ is $\binom{n-1}{k-1}$ and that any independent set of maximum size must consist of the $k$-subsets containing some given element of $[n]$. This theorem is widely seen as a cornerstone of extremal combinatorics and has connections to other areas of mathematics including representation theory, algebraic combinatorics, and spectral graph theory \cite{BollobasComb,GodsilMeagher}. Lov\'{a}sz \cite{MR514625} used topological methods to prove that the  chromatic number of Kneser graphs is $n-2k+2$ for $n\ge 2k$ (see also \cite{MR514626, greene2002new,matouvsek2004combinatorial}). There are also several interesting results proving that the Kneser graphs and bipartite Kneser graphs are Hamiltonian in certain ranges of parameters \cite{Chen, MR3826304,MR3759914}, but this problem is not completely solved for all parameters $n$ and $k$.
    
    
    

    

    In this paper, we investigate the \textit{toughness} of Kneser graphs, which is a measure of graph connectivity. The toughness of a connected graph $G$ is defined as
    $$t(G)= \min_{S} \frac{|S|}{c(G\setminus S)}, $$
    where $S$ ranges over all vertex cuts of $G$, and $c(G\setminus S)$ denotes the number of components remaining in $G$ after deleting $S$. We will use $G\setminus S$ to denote both the resulting graph and its corresponding vertex set. No confusion will arise. A graph $G$ is called \emph{$t$-tough} if $t(G)\ge t$. 
    
    Chv\'atal \cite{chvatal} introduced the study of toughness in connection with the cycle structure of a graph. He observed that every Hamiltonian graph is $1$-tough and conjectured that there exists some $t$ such that any $t$-tough graph is Hamiltonian. Bauer, Broersma and Veldman  \cite{notEveryTwotough} later showed that if such a $t$ exists, then it must be at least $\frac{9}{4}$. Though Chv\'atal's conjecture remains open, it has spurred significant work on the close connection between Hamiltonicity and toughness \cite{survey}. 

    Let $\ell=\lambda_1> \lambda_2\ge \cdots \ge \lambda_N$ denote the eigenvalues of the adjacency matrix of a connected $\ell$-regular graph $G$ on $N$ vertices.  Brouwer~\cite{MR1344566} (see also Alon \cite{AlonTough} for related results) showed that if $\lambda = \max \{|\lambda_2|, |\lambda_N|  \}$, then 
    $$t(G) > \frac{\ell}{\lambda}-2.$$
    This eigenvalue bound led Alon \cite{AlonTough} to disprove a conjecture of Chv\'{a}tal that a graph of sufficiently large toughness must be pancyclic (contain a cycle of every length). Brouwer \cite{Brouwer96} conjectured that the bound above can be improved to $t(G) \geq \ell / \lambda-1$, which was recently verified by Gu \cite{gu2020proof}. He also stated his belief that many interesting $\ell$-regular graphs satisfy $t(G) = \ell / \lambda$. We explain the intuition behind this assertion below. The following result is known as the Hoffman ratio bound for the independence number of an $\ell$-regular graph (see Chapter 9 of \cite{GodsilMeagher} or \cite{hoffmanratiobound} for more information).
    \begin{thm} [Hoffman \cite{hoffmanratiobound}]
    Let $G$ be a connected $\ell$-regular graph on $N$ vertices with eigenvalues $\ell=\lambda_1> \lambda_2\ge \cdots \ge \lambda_N$. If $\alpha(G)$ denotes the independence number of $G$, then
    $$
    \alpha(G)\leq \frac{N|\lambda_N|}{\ell+|\lambda_{N}|}.
    $$
    \end{thm}
    Since $|\lambda_N|\leq \lambda$, the Hoffman ratio bound implies that $\alpha(G)\leq \frac{N\lambda}{\ell+\lambda}.$ If equality occurs and $Q$ is an independent set of maximum size, then the complement of $Q$ is a vertex cut in $G$ whose removal creates $\frac{N\lambda}{\ell+\lambda}$ singletons. A simple calculation yields that $t(G)\leq \ell/\lambda$.  Cioab\u{a} and Wong \cite{MR3240840} confirmed Brouwer's intuition for several classes of regular graphs attaining equality in the Hoffman ratio bound including complements of point graphs of generalized quadrangles, lattice graphs ($2$-dimensional Hamming graphs), Kneser graphs $K(n,2)$ and their complements,  with the sole exception of the Petersen graph $K(5,2)$.
    
    In this paper, we extend these results and determine $t(K(n,3))$ for every $n\geq 7$ and $t(K(n,4))$ for every $n\geq 9$. In the case of Kneser graphs $K(n,k)$, $\ell/\lambda=n/k-1$ (see Theorem \ref{knesereigenvalue}).
    \begin{thm}\label{smallk}
    Let $k\in \{3,4\}$. The toughness of the Kneser graph $K(n,k)$ equals $$t(K(n,k)) = \frac{n}{k} - 1,$$ for any $n\geq 2k+1$. Moreover, any subset of vertices $S$ satisfying $t(K(n,k))=\frac{|S|}{c(K(n,k)\setminus S)}$ must be the complement of a maximum independent set in $K(n,k)$.
    \end{thm}
    
    We also prove that for  $k\geq 5$, $t(K(n,k))=n/k-1$ for $n$ sufficiently large as a function of $k$.
    \begin{thm}{\label{quad}}
    Let $k\geq 5$ be a natural number. If $$n \geq \frac{2}{\ln 2} k^2 + \left(2 - \frac{3}{\ln 2}\right) k + \frac{1}{\ln 2} \geq \frac{2^{\frac{1}{k-1}}(2k)-1}{2^{\frac{1}{k-1}}-1},$$ then
    $$t(K(n,k)) = \frac{n}{k} -1.$$
    Any subset of vertices $S$ such that $t(K(n,k))=\frac{|S|}{c(K(n,k)\setminus S)}$ must be the complement of a maximum independent set in $K(n,k)$.
    \end{thm}

    
    

    
    We believe that the toughness of Kneser graphs is exactly ${\ell}/{\lambda}=n/k-1$. Using the explicit expression for the spectrum of Kneser graphs from Theorem \ref{knesereigenvalue}, we obtain the following conjecture.
    \begin{conj}
    Let $k\geq 5$ and $n\geq 2k+1$. The Kneser graph $K(n,k)$ has toughness
    $$t(K(n,k)) =  \frac{n}{k} -1.$$
    If $S$ is a subset of vertices of $K(n,k)$ such that  $t(K(n,k))=\frac{|S|}{c(K(n,k)\setminus S)}$, then $S$ is the complement of a maximum independent set in $K(n,k)$.
    \end{conj}

    In Section \ref{sec2} we will collect some important spectral bounds on vertex cuts and results on extremal intersecting families. In Section \ref{sec3} we will prove Theorem \ref{quad}, and in Sections \ref{sec4} and \ref{sec5} we will prove Theorem \ref{smallk}. Sections \ref{sec3} and \ref{sec4} follow from an author's Ph.D. thesis \cite{wongthesis}. Recent developments have allowed us to shorten and improve the presentation.
    
    
    \section{Tools}\label{sec2}
    
    For two disjoint subsets $S,T$ of the vertices of a graph $G$, we denote by $e(S,T)$ to be the number of edges with one endpoint in $S$ and the other in $T$. Also, we use $e(S)$ for the number of edges with both endpoints in $S$. Determining the toughness for small $n$ values will require a mix of theoretical and computational arguments, such as  counting $e(S,K(n,k)\setminus S)$ in two ways. Classic results on intersecting families provide an upper bound on $c(K(n,k)\setminus S)$ and thus $|S|$.
    We then apply a spectral lower bound on the size of a vertex cut to either obtain a contradiction or restrict the possible values of $|S|$, which will be eliminated through other techniques.

    \subsection{Spectral bounds}
    We will frequently use the spectrum of the adjacency matrix of the Kneser graph.
    
    \begin{thm}\cite[Ch. 9]{AGT}{\label{knesereigenvalue}}
    The eigenvalues of the adjacency matrix of $K(n,k)$ are $(-1)^{j}\binom{n-k-j}{k-j}$ with multiplicities $\binom{n}{j}-\binom{n}{j-1}$ for $j=0, ..., k$.
    \end{thm}
    
    By definition, the Laplacian eigenvalues of $K(n,k)$ can be calculated by subtracting the above eigenvalues from the degree of regularity $\binom{n-k}{k}$. The following provides a bound on $e(S,K(n,k)\setminus S)$.
    
    \begin{lem}[Mohar \cite{mohar1997some}]{\label{laplacian}}
    Let $G$ be a connected graph of order $n$ and $T$ be a subset of vertices of $G$. Let $0=\mu_1<\mu_2\le ...\le \mu_n$ denote the eigenvalues of the Laplacian matrix of $G$. Then
    $$\frac{\mu_2|T|(n-|T|)}{n} \leq e(T, G\setminus T) \leq \frac{\mu_{n}|T|(n-|T|)}{n}.$$
    \end{lem}
    
    The following result was proved independently by Haemers \cite{haemers1995interlacing} and Helmberg, Mohar, Poljak and Rendl \cite{helmberg1995spectral}.
    \begin{thm}[\cite{haemers1995interlacing, helmberg1995spectral}]{\label{HMP}} 
    Let $G$ be a connected $\ell$-regular graph on $N$ vertices, and let $\ell=\lambda_1> \lambda_2\ge \cdots \ge \lambda_N$ denote the eigenvalues of the adjacency matrix.  If $G\setminus S$ separates into vertex sets $S_{1}$ and $S_{2}$, then
    $$|S| \ge \frac{4(\ell-\lambda_2)(\ell-\lambda_{N})|S_1||S_2|}{N(\lambda_{2}-\lambda_N)^{2}}.$$
    \end{thm}
    
    Using the explicit expression for the spectrum of $K(n,k)$ in Theorem \ref{knesereigenvalue}, this bound specializes to
    \begin{align}
    &|S| \ge \frac{8(n-4)(n-6)|S_1||S_2|}{3(n-2)^{3}}, \hspace{1cm} k=3, \label{HMP1}\\
    &|S| \ge \frac{6(n-8)(n-5)|S_1||S_2|}{(n-3)(n-2)^{3}}, \hspace{1cm} k=4.\label{HMP2}
    \end{align}
    
    \subsection{Vertex Partitions}
    In order to apply Theorem \ref{HMP}, we need to guarantee the existence of two vertex sets $S_1,S_2$ with orders as close to equal as possible, in order to maximize the quadratic bound. This is the purpose of Lemma \ref{minlemma}, which can easily be proven by induction on $c$. Details can be found in \cite[Lemma 5.4.2]{wongthesis}.
    \begin{lem}{\label{minlemma}}
    Let $c\ge 2$ and $1\le n_{1}\le n_{2}\le...\le n_{c}$ be integers.  If $\sum_{i=1}^{c}n_{i}\ge 2c$, then there exists a partition 
    $S_1 \cup S_2 = [c], S_1 \cap S_2 = \emptyset$
    such that 
    $$\min\left(   \sum_{i \in S_1}n_i, \sum_{i \in S_2} n_i \right) \geq c-1. $$
    \end{lem}

    \subsection{Intersecting Families}
    
    The classic Erd\"os-Ko-Rado theorem gives an upper bound on the size of an intersecting  family. 
    \begin{thm}[Erd\"os-Ko-Rado \cite{erdos}]\label{EKR}
    Let $n\ge 2k \ge 2$.  If $\F\subseteq \binom{[n]}{k}$ is intersecting, then 
    $$|\F| \le \binom{n-1}{k-1}.$$
    Equality occurs if and only if $\F=\F(x)$ for some $x\in [n]$, where 
    \[\F(x)=\{A\in \binom{[n]}{k}:x\in A\}.\]
    \end{thm}
    Intersecting families of the form $\F(x)$ are called trivial. The next significant result in this direction is the Hilton-Milner theorem showing that nontrivial intersecting set families are significantly smaller.
    \begin{thm}[Hilton-Milner \cite{hm}]\label{HM}
    Let $n\ge 2k+1 \ge 5$. If $\F\subseteq \binom{[n]}{k}$ is intersecting and $\cap_{F\in\F}F=\emptyset$, then $$|\F| \le \binom{n-1}{k-1}-\binom{n-k-1}{k-1}+1.$$ Furthermore, for $k\ne 3$, equality occurs if and only if $\F=\F(A,x)$ for some $A\in \binom{[n]}{k}$ and $x\notin A$, where
    $$\F(A,x) = \{A\} \cup \{B\in \binom{[n]}{k} :\, x\in B, B\cap A \ne \emptyset\}.$$  
    If $k =3$, equality occurs if and only if $\F=\F(A,x)$ as above or $\F=\mathcal{G}(A)$ for some $A\in \binom{[n]}{k}$, where
    $$\mathcal{G}(A) = \{B\in \binom{[n]}{k}: \, |A \cap B|\ge 2\}.$$  
    \end{thm}
    
    We define the \textit{diversity} $\gamma(\F)$ of an intersecting family $\F$ as follows: 
    $$\gamma(\F) = |\F| - \max_{i\in [n]} |\{ X \in \F: i \in X \}|.$$
    All $i \in [n]$ that maximize $|\{ X \in \F: i \in X \}|$ are referred to as \textit{central} in $\F$.
    
    Note that $\gamma(\F(x))=0$ (defined as in Theorem \ref{EKR}) and  $\gamma(\F(A,x))=1$ (as in Theorem \ref{HM}). The following theorem focuses on families of higher diversity, and also includes a characterization of the extremal families which we omit.
    
    \begin{thm}[Han--Kohayakawa \cite{han}]\label{han}
    Let $n\geq 2k+1 \geq 7$. Suppose $\F \subseteq \binom{[n]}{k}$ is intersecting and $\F$ is not a subfamily of $\F(x),\F(A,x)$, or $\mathcal{G}(A)$. Then
    $$|\F| \leq   \binom{n-1}{k-1} - \binom{n-k-1}{k-1} - \binom{n-k-2}{k-2}+2.$$
    \end{thm}
    
    Another celebrated result from extremal set theory is the Kruskal-Katona theorem, which provides a lower bound on the number of particular subsets of a family. 
    \begin{thm}[Kruskal--Katona \cite{kruskal, katona}]\label{KK}
    Let $n\geq i\geq r$ be natural numbers and $\F \subseteq \binom{[n]}{i}$. Denote by $\F_1$ the family of all $(i-r)-$size subsets of the sets in $\F$, 
    $$\F_1=\{B\in \binom{[n]}{i-r}: B\subseteq A, \text{ for some } A\in \F\}.$$
    If we uniquely expand 
    $$|\F| = \binom{n_i}{i} + \binom{n_{i-1}}{i-1} + \cdots + \binom{n_j}{j},$$ where $n_i > n_{i-1} > \ldots > n_j \ge j \ge 1$ are natural numbers, then
    $$|\F_1| \ge \binom{n_i}{i-r} + \binom{n_{i-1}}{i-r-1} + \cdots + \binom{n_j}{j-r}.$$ 
    \end{thm}
    
    For a graph $G=(V,E)$ and a subset of vertices $Q\subseteq V$, the \emph{neighborhood} $N(Q)$ consists of the vertices not in $Q$ that are adjacent to a vertex in $Q$.
    \begin{lem}\label{KKBounds}
    Let $n \ge 2k + 1 \ge 5$ and let $Q$ be an independent set in $K(n,k)$. If we uniquely expand
    $$|Q| = \binom{q_{n-k}}{n-k} + \binom{q_{n-k-1}}{n-k-1} + \cdots + \binom{q_j}{j},$$ where $q_{n-k} > q_{n-k-1} > \ldots > q_j \ge j \ge 1$ are natural numbers, then $$|N(Q)| \ge \binom{q_{n-k}}{k} + \binom{q_{n-k-1}}{k-1} + \cdots + \binom{q_j}{j-(n-2k)}.$$  
    \end{lem}
    \begin{proof}
    
    Define $\F = \{A^c : A \in Q\},$ where $A^c$ $=[n]\setminus A$. Note $\F \subseteq {[n] \choose n-k}$, and clearly for all $X \in \F$, we have that $X$ is disjoint from $A \in Q$ for some $A$, namely the $A$ such that $X=A^c$. Therefore, all $k$-size subsets of $X$ are in $N(Q)$. Then if $\F_1$ is the family of $k$-size subsets of the elements of $\F$, $\F_1 \subseteq N(Q)$ and thus $|\F_1|\le|N(Q)|$. Note that $|\F|=|Q|$. The result then follows from Theorem \ref{KK} with $i=n-k$ and $r=n-2k$.
    \end{proof}
    
    
    
    \begin{lem}\label{switch}
    Suppose $n \ge 2k + 1 \ge 5$. Let $Q$ be an independent set of vertices in $K(n,k)$.  Suppose $x_1x_2$ is an edge such that $x_1\in Q$.  If $N(\{x_1,x_2\})$ and $Q$ are disjoint, then $|Q| \le \binom{n-1}{k-1}-\binom{n-k-1}{k-1}+1$.
    \end{lem}
    
    \begin{proof}
    Assume to the contrary that $|Q| \ge\binom{n-1}{k-1}-\binom{n-k-1}{k-1} + 2$. By  the contrapositive of Theorem \ref{HM}, we must have $\gamma(Q) =0$.  Let $Q'=Q\cup\{x_2\}\setminus\{x_1\}$.  Again we have  $\gamma(Q') = 0$. Since $x_1$ and $x_2$ are adjacent and thus disjoint, $Q$ and $Q'$ cannot share a central element.  Let $y_1$ be a central element of $Q$.  Then $y_1$ cannot be an element of $x_2$, and so there exists another central element $y_2$ in $Q'$.   Thus, all elements in $Q \setminus \{x_1\}$ must contain $y_1$ and $y_2$, both central elements. This gives us that $|Q| \le \binom{n-2}{k-2}+1$. However,
    $$\binom{n-1}{k-1}-\binom{n-k-1}{k-1}+2=\binom{n-2}{k-2}+\binom{n-2}{k-1}-\binom{n-k-1}{k-1}+2>\binom{n-2}{k-2} +1,$$
    a contradiction.
    \end{proof}
    
    \begin{lem}\label{switch2}
    Suppose $n \ge 2k + 1 \ge 9$.  Let $Q$ be an independent set of vertices in $K(n,k)$.  Suppose $x_1,x_2,x_3,x_4$ induce a 4-cycle, and that $x_1,x_3\in Q$.  If $N(\{x_1, x_2,x_3,x_4\})$ and $Q$ are disjoint, then 
    \begin{equation}\label{lemma14}
        |Q| \leq    \binom{n-1}{k-1} - \binom{n-k-1}{k-1} - \binom{n-k-2}{k-2} + 2.
    \end{equation}
    Moreover, if $x_1,x_2,x_3,x_4$ instead induce two $K_2$'s, then we again have (\ref{lemma14}).
    \end{lem}
    \begin{proof} 
    Assume to the contrary that $|Q|\ge \binom{n-1}{k-1} - \binom{n-k-1}{k-1} - \binom{n-k-2}{k-2}+3.$  By the contrapositive of Theorem \ref{han} and with $k\ge 4$ by assumption, $Q$ is a set of the form $\mathcal{F}(x)$ or $\mathcal{F}(A,x)$.   Let $Q'=Q\cup\{x_2,x_4\}\setminus \{x_1,x_3\}$.  By the contrapositive of Theorem \ref{han}, $\gamma(Q),\gamma(Q')\le 1$.  Let $y_1$ be a central element of $Q$.  Since $x_1$, $x_3$ are adjacent to both $x_2$ and $x_4$, $y_1$ cannot be a central element of $Q'$.  So there must be some $y_2$ as a central element of $Q'$. We find that for all $A\in Q\setminus\{x_1,x_3\}$, necessarily $y_1,y_2\in A$.  Thus  $|Q|\le \binom{n-2}{k-2}+2$.  However, noting that $n \geq 2k+1 \geq 9$, we have 
    \begin{align*}{\label{inequal}}
    \binom{n-1}{k-1} - &\binom{n-k-1}{k-1} - \binom{n-k-2}{k-2} + 3\\ &=\binom{n-2}{k-2}+\binom{n-2}{k-1} - \binom{n-k-1}{k-1} - \binom{n-k-2}{k-2}+3\\
    &=\binom{n-2}{k-2} + \left(\binom{n-3}{k-1}-\binom{n-k-1}{k-1}\right)+ \left(\binom{n-3}{k-2}-\binom{n-k-2}{k-2}\right)+3\\
    &>\binom{n-2}{k-2}+2,
    \end{align*}
    a contradiction.
    
    Now suppose $x_1x_4$ and $x_2x_3$ are two induced edges.  Assume to the contrary that $|Q|\ge \binom{n-1}{k-1} - \binom{n-k-1}{k-1} - \binom{n-k-2}{k-2}+3$, with $x_1,x_3\in Q$.   By the contrapositive of Theorem \ref{han}, we have $\gamma(Q)\le 1$.   Let $y_1$ be a central element in $Q$. We see that $y_1\in x_1$ or $y_1\in x_3$.  Furthermore, if $y_1\in x_1$, then $y_1\not\in x_4$ due to adjacency.  Similiarly, if $y_1\in x_3$, then $y_1\not\in x_2$.   It follows that there exists a $z_1\in \{x_1,x_4\}$ and $z_2\in\{x_2,x_3\}$ such that $y_1\not\in z_1$ and $y_1\not\in z_2$. Define $H=Q\setminus\{x_1,x_3\}$ and  $Q''=H\cup\{z_1,z_2\}$. Now $|Q|=|Q''|$, so by the contrapositive of  Theorem \ref{han}, $\gamma(Q'')\le 1$, and thus $y_1$ cannot be a central element in $Q''$.  So there exists some other central element $y_2$ in $Q''$.  Hence, every element in $H$ must contain $y_1$ and $y_2$, implying $|Q|\le \binom{n-2}{k-2}+2$. By the inequality above, we again have a contradiction.
    \end{proof}

    \subsection{Structural Properties}
    
    Some properties of $K(n,k)$ follow as a consequence of the Kneser graph being vertex and edge-transitive (\cite[Ch. 3]{AGT}, \cite{MR2828414}, and \cite{WATKINS197023}).

    \begin{lem}{\label{discon}}
    For any vertex $u$ in $K(n,k)$, the graph obtained by deleting $u$ and its neighborhood is connected.  
    \end{lem}

    \begin{lem}{\label{toughone}}
    Let $G$ be a non-bipartite $\ell$-regular graph.  Assume the edge-connectivity of $G$ is $\ell$, and the only edge cuts of $\ell$ edges are the $\ell$ edges incident to a vertex. Then for any vertex cut $S$, $|S| > c(G\setminus S)$. In particular, $t(K(n,k))>1$ for any $n\ge 2k+1$.  
    \end{lem}
    \begin{proof}
    We have $ c(G\setminus S)\ell \le e(S, G \setminus S)$ for any vertex cut $S$, since the edge-connectivity is equal to the degree $\ell$. Equality occurs if and only if the components of $G\setminus S$ are singletons.  Furthermore, $e(S, G \setminus S) \le \ell|S|$, with equality occurring if and only if $S$ is an independent set. Combining yields $\frac{|S|}{c(G\setminus S)} \ge 1$, with equality if and only if $G$ is bipartite.  The last statement holds since the edge-connectivity of $K(n,k)$ is $\ell$ (see \cite[Ch. 3]{AGT} or \cite{MR2567957}) and $\chi(K(n,k)) \ge 3$.
    \end{proof}
    \begin{lem}\label{trees}
    Let $S$ be a vertex cut achieving the toughness of $K(n,k)$. Then every component in $K(n,k)\setminus S$ is either a singleton, $K_2$, or is biconnected.
    \end{lem}
    \begin{proof}
    The previous lemma guarantees that $$ \frac{|S|}{c(K(n,k)\setminus S)} > 1.$$   Let $\C$ be a component in $K(n,k)\setminus S$ on at least 3 vertices.  Assume to the contrary that deleting vertex $v$ in $\C$ disconnects $\C$.   Then by adding $v$ to $S$,
    $$\frac{|S \cup \{v\}|}{c(K(n,k)\setminus (S \cup \{v\}))} \leq \frac{|S |+1}{c(K(n,k)\setminus S)+1} < \frac{|S|}{c(K(n,k)\setminus S)},$$
    contradicting $S$ achieving the toughness.
    \end{proof}
    
    In particular, Lemma \ref{trees} implies that $K(n,k)\setminus S$ cannot contain a tree on more than two vertices. 
    
    
    \section{Quadratic upper bound}\label{sec3}
    We will first show that our toughness conjecture holds for fixed $k$ and sufficiently large $n$. 
    
    \begin{thm}\label{TK}
    Let $n\geq 2k+1$ and $k \geq 3$. Let $S$ be a vertex cut of $K(n,k)$ and $c=c(K(n,k)\setminus S)$. Then either \[\binom{n-k-1}{k-1}+1\leq c\leq \binom{n-1}{k-1}-\binom{n-k-1}{k-1}+1,\]
    or $$\frac{|S|}{c}\geq \frac{n}{k}-1,$$ with equality if and only if $S$ is the complement of a maximum independent set.  
    \end{thm}  
    \begin{proof}
    Let $S$ be a vertex cut and $c = c(K(n,k)\setminus S)$. Assume first that $c\le \binom{n-k-1}{k-1}$. If $|S| = \binom{n-k}{k}$, then structural properties mentioned above tell us that $S = N(u)$ for some vertex $u$, and so Lemma \ref{discon} implies 
    $$\frac{|S|}{c(K(n,k)\setminus S)} =\frac{\binom{n-k}{k}}{2} >\frac{n}{k}-1$$ 
    for $n\ge2k+1$ and $k\ge3$.   If instead  $|S|>\binom{n-k}{k}$, then
    $$\frac{|S|}{c(K(n,k)\setminus S)} >\frac{\binom{n-k}{k}}{c} \ge\frac{n}{k}-1,$$
    and we are done. Now assume $c> \binom{n-1}{k-1}-\binom{n-k-1}{k-1}+1$. If $K(n,k)\setminus S$ contains a non-singleton component, then choosing one vertex from each component to be in $Q$ gives a contradiction by Lemma \ref{switch}.
    Hence by Theorem \ref{EKR}, $K(n,k)\setminus S$ contains at most $\binom{n-1}{k-1}$ singletons. Then
 $$\frac{|S|}{c} \ge \frac{\binom{n}{k}-\binom{n-1}{k-1}}{\binom{n-1}{k-1}} =\frac{\binom{n}{k}}{\binom{n-1}{k-1}} - 1 = \frac{n}{k}-1.$$
    Note that equality can only occur when $c> \binom{n-1}{k-1}-\binom{n-k-1}{k-1}+1$, in which case it follows by Theorem \ref{EKR} that $S$ is the complement of a maximum independent set of size $\binom{n-1}{k-1}$.
    \end{proof}
    
    \begin{cor}
    For $k\geq 3$ and 
    \[n > \frac{2}{\ln(2)}k^{2}+\left(2-\frac{3}{\ln(2)}\right)k+\frac{1}{\ln(2)} \geq \frac{2^{\frac{1}{k-1}}(2k)-1}{2^{\frac{1}{k-1}}-1},\]
    we have
    $$t(K(n,k))=\frac{n}{k}-1.$$ Moreover, if $S$ is a vertex cut of $K(n,k)$ such that $\frac{|S|}{c(K(n,k)\setminus S)}=\frac{n}{k}-1$, then $S$ is the complement of a maximum independent set. 
    \end{cor}
    \begin{proof}
    For $k\geq 3$, we have
    \[n>\frac{2^{\frac{1}{k-1}}(2k)-1}{2^{\frac{1}{k-1}}-1}.\]
    This is equivalent to
    \begin{equation}{\label{equ1}}
     2(n-2k)^{k-1} > (n-1)^{k-1}.
    \end{equation}
    For $k \ge 3$,
    $$2\binom{n-k-1}{k-1} > \frac{2(n-2k)^{k-1}}{(k-1)!} \text{\hspace{.2in} and\hspace{.2in}}  \frac{(n-1)^{k-1}}{(k-1)!} > \binom{n-1}{k-1}.$$
    Equation \eqref{equ1} now implies
    $$2\binom{n-k-1}{k-1} > \binom{n-1}{k-1}.$$ 
    This is equivalent to 
    $$ \binom{n-k-1}{k-1}+1 > \binom{n-1}{k-1}-\binom{n-k-1}{k-1}+1.$$
    Note these are the bounds for $c$ in Theorem \ref{TK}. It can be verified using a two term Taylor approximation that
    \[\frac{2}{\ln(2)}k^{2}+\left(2-\frac{3}{\ln(2)}\right)k+\frac{1}{\ln(2)} \geq \frac{2^{\frac{1}{k-1}}(2k)-1}{2^{\frac{1}{k-1}}-1},\] finishing the proof.
    \end{proof}
    

    
    \section{Toughness of \texorpdfstring{$K(n, 3)$}{K(n,3)}}\label{sec4}
    
    \begin{thm}
    For $n\ge9$, $t(K(n,3)) =\frac{n}{3}-1$. Moreover, if $S$ is a vertex cut of $K(n,3)$ such that $\frac{|S|}{c(K(n,3)\setminus S)} = \frac{n}{3}-1$,  then $S$ is the complement of a maximum independent set of $K(n,3)$.
    \end{thm}
    
    \begin{proof}
    First, the bounds for $c$ in Theorem \ref{TK} do not agree for $n\geq 12$. Consider when $8\le n\le 11$. In all cases, we assume there exists a vertex cut $S$ that is not the complement of a maximum independent set such that $\frac{|S|}{c(K(n,3)\setminus S)}\leq \frac{n}{3}-1$ to obtain a contradiction.
    
    For $n=11$, by Theorem \ref{TK}, we have that $22\le c(K(11,3)\setminus S)\le 25$. Therefore, $\frac{|S|}{c(K(11,3)\setminus S)}\le\frac{8}{3}$ implies $56 \le |S|\le 66$. If $S$ is the neighborhood of a vertex, then $c(K(11,3)\setminus S)=2$, contradicting the lower bound on $c(K(11,3)\setminus S)$. Therefore, $S$ is not the neighborhood of a vertex, and thus by Lemma \ref{discon}, $|S|\neq 56$. Thus $57 \le |S|\le 66$. Furthermore, we have  $|K(11,3)\setminus S| =165-|S|\ge 165-66 =99>2(25)\ge 2c(K(11,3)\setminus S) $, so we can apply Lemma \ref{minlemma}.  There exist subsets of vertices $S_{1}$ and $S_{2}$ of $K(11,3)\setminus S$ with no edges between them, such that $\min(|S_1|,|S_2|)\ge 21$. Equation \eqref{HMP1} implies
    $$|S|\ge \frac{280}{2187}|S_1||S_2| \ge \frac{280}{2187}(21)(99-21) > 209,$$
    which is a contradiction with $|S|\le 66$.  
    
    The graph $K(10,3)$ is 35-regular and has 120 vertices. By Theorem \ref{TK}, we have that $16\le c(K(10,3)\setminus S)\le 22$. As $S$ is not the neighborhood of a vertex,  $\frac{|S|}{c(K(10,3)\setminus S)}\le\frac{7}{3}$ implies $36\le |S|\le 51$.  Furthermore, we have  $|K(10,3)\setminus S| =120-|S|\ge 120-51 =69 >2(22)\ge 2c(K(10,3)\setminus S) $, so we can apply Lemma \ref{minlemma}. There exist subsets of vertices $S_{1}$ and $S_{2}$ of $K(10,3)\setminus S$ with no edges between them, such that $\min(|S_1|,|S_2|)\ge 15$. Equation \eqref{HMP1} implies
    $$|S|\ge \frac{|S_1||S_2|}{8} \ge \frac{1}{8}(15)(69-15) >101,$$
    a contradiction with $|S|\le 51$.
    
    The graph $K(9,3)$ is 20-regular and has 84 vertices. By Theorem \ref{TK}, we have that $11\le c(K(9,3)\setminus S)\le 19$. As $S$ is not the neighborhood of a vertex,  $\frac{|S|}{c(K(9,3)\setminus S)}\le 2$ implies $21\le |S|\le 38$.  Furthermore, we have  $|K(9,3)\setminus S| =84-|S|\ge 84-38 =46 >2(19)\ge 2c(K(9,3)\setminus S) $, so we can apply Lemma \ref{minlemma}. There exist subsets of vertices $S_{1}$ and $S_{2}$ of $K(9,3)\setminus S$ with no edges between them, such that  $\min(|S_1|,|S_2|)\ge 10$. Equation \eqref{HMP1} implies
    $$|S|\ge \frac{40}{343}|S_1||S_2| \ge \frac{40}{343}(10)(46-10) > 41,$$
    a contradiction with $|S|\le 38$.
    \end{proof}

    For $K(8,3)$, we must split into casework on $c$.
    
    \begin{thm}
    We have $t(K(8,3)) =\frac{5}{3}$. Moreover, if $S$ is a vertex cut of $K(8,3)$ such that $\frac{|S|}{c(K(8,3)\setminus S)} = \frac{5}{3}$,  then $S$ is the complement of a maximum independent set of $K(8,3)$.
    \end{thm}
    
    \begin{proof}
    The graph $K(8,3)$ is 10-regular and has 56 vertices. Assume there exists a vertex cut $S$ achieving toughness such that $S$ is not the complement of a maximum independent set and $\frac{|S|}{c(K(8,3)\setminus S)}\leq \frac{5}{3}$. By Theorem \ref{TK}, we have that $7\le c(K(8,3)\setminus S)\le 16$. As $S$ is not the neighborhood of a vertex,  $\frac{|S|}{c(K(8,3)\setminus S)}\le\frac{5}{3}$ implies $11\le |S|\le 26$. 
    
    
    If $c(K(8,3)\setminus S)=7$, then $|S|=11$. Furthermore, we have $|K(8,3)\setminus S| = 56-11= 45>2(7) = 2c(K(8,3)\setminus S)$, so we can apply Lemma \ref{minlemma}. There exist subsets of vertices $S_{1}$ and $S_{2}$ of $K(8,3)\setminus S$ with no edges between them, such that $\min (|S_1|, |S_2|) \ge 6$. Equation \eqref{HMP1} implies
    $$|S|\ge \frac{8}{81}|S_1||S_2| \ge \frac{8}{81}(6)(45-6) >  23,$$
    a contradiction with $|S|=11$. A similar argument shows a contradiction for $8 \le c(K(8,3)) \le 14$.
    
    
    
    If $c(K(8,3)\setminus S)=15$, then $16\le |S|\le 25$.
    Let $a$ and $b$ denote the number of singletons and $K_2$'s in $K(8,3)\setminus S$, respectively. Let $\C$ be a component of $K(8,3)\setminus S$. Note the girth of $K(8,3)$ is $4$ \cite{girth_ref}. If $|\C|=1,2,4$, then  $e(\C,S)\ge 10,18,32$. By Lemma \ref{trees}, there are no components of order $3$. As each component must contain at least one vertex, we find that $|\C|\le |K(8,3)\setminus S|-14\le 26$. The lower bound in Lemma \ref{laplacian} is increasing for $|\C|\leq 28$. Thus if $|\C|\geq 4$, then \[e(\C,S)\geq \ceil*{\dfrac{7\cdot 4(56-4)}{56}}\geq 26.\]
    Since there are $10$ edges from each singleton, $18$ edges from each $K_2$, and at least $28$ edges for each other component, we have
    \[10a+18b+26(15-a-b)\leq e(S,K(8,3)\setminus S)\leq \floor*{\frac{16|S|(56-|S|)}{56}}\le 221.\]
    Finally, if $a\geq 7$,  Lemma \ref{KKBounds} implies the neighborhood contains at least 26 vertices, contradicting $|S|\leq 25$. This reduces to the system
    \[\begin{cases}
        0\leq a\leq 6,\\
        0\leq b\leq 15-a,\\
        10a+18b+26(15-a-b)\leq 221,
    \end{cases}
    \]
    which has no integer solutions.
    
    If $c(K(8,3)\setminus S)=16$, then $17\le |S|\le 26$. Let $a,b$ denote the number of singletons and $K_2$'s in $K(8,3)\setminus S$. By a similar analysis of $e(\C,S)$ for a component $\C$ of $K(8,3)\setminus S$, we have the system,
    \[\begin{cases}
        0\leq a\leq 8,\\
        0\leq b\leq 16-a,\\
        10a+18b+26(16-a-b)\leq 222,
    \end{cases}
    \]
    which has no integer solutions.\end{proof}
    
    For $K(7,3)$, we appeal to an edge-counting argument and computer search.
    
    \begin{thm}
    We have $t(K(7,3))=\frac{4}{3}$. Moreover, if $S$ is a vertex cut of $K(7,3)$ such that $\frac{|S|}{c(K(7,3)\setminus S)} = \frac{4}{3}$,  then $S$ is the complement of a maximum independent set of $K(7,3)$.
    \end{thm}
    
    \begin{proof}
    
    Assume there exists a vertex cut $S$ achieving toughness such that $S$ is not the complement of a maximum independent set and $\frac{|S|}{c(K(7,3)\setminus S)}\leq \frac{4}{3}$. By Theorem \ref{TK}, we have that $5\le c(K(8,3)\setminus S)\le 13$. By Lemma \ref{toughone}, $\frac{|S|}{c(K(7,3)\setminus S)}\le\frac{4}{3}$ implies $6 \le |S| \le 17$. 
    
    Since $K(7,3)$ is $4$-regular, we have that $e(S,K(7,3)\setminus S)\leq 4|S|$. By Lemma \ref{laplacian}, we may improve this upper bound to
    \begin{equation}\label{7,3:ehigh}
        e(S,K(7,3)\setminus S)\leq \min\left(4|S|,\left\lfloor\dfrac{7|S|(35-|S|)}{35}\right\rfloor\right).
    \end{equation}
    Now for any component $\C$ of $K(7,3)\setminus S$, we will bound $e(\C,S)$ from below. By Lemma \ref{laplacian}, we have
    \[
        e(\C,S)\geq \dfrac{2|\C|(35-|\C|)}{35}.
    \]
    We will require tighter bounds for small $|\C|$. When $|\C|=1,2$, we trivially obtain $e(\C,S)=4,6$. Since $K(7,3)$ has girth $6$ \cite{girth_ref}, if $3\leq |\C|\leq 5$, then $\C$ is a tree violating Lemma \ref{trees}. Hence we do not consider when $3\leq |\C|\leq 5$. For $|\C|=6,7$, $e(\C)$ is maximized when $\C$ is a cycle. Since $K(7,3)$ is $4$-regular,
    \[e(\C,S)=4|\C|-2e(\C).\]
    Thus $e(\C,S)\geq 12,14$ for $|\C|=6,7$.   We combine this into a single function $f_1(|\C|)$ which provides a lower bound for $e(\C,S)$. Define
    \begin{equation}\label{7,3:f}
    f_1(t) = \begin{cases}
    4 & \text{if }t=1\\
    6 & \text{if }t=2\\
    12 & \text{if }t=6\\
    14 & \text{if }t=7\\
    \ceil{\frac{2t(35-t)}{35}} & \text{if } t\ge 8
    \end{cases}.\end{equation}
    
     Let $X_1,\ldots,X_c$ denote the components of $K(7,3)\setminus S$. Let $p_i=|X_i|$, and without loss of generality assume $p_1\leq \cdots \leq p_c$. Let $\mathcal{P}(m, j, T)$ denote the set of integer partitions of $m$ into $j$ parts such that the size of each part is not an element of $T$.
     By \eqref{7,3:ehigh} and \eqref{7,3:f}, $(p_1,\ldots,p_c)$ is a partition in $\mathcal{P}(35-|S|,c,\{3,4,5\})$ satisfying
     
    \begin{equation}\label{7,3:check}\sum_{i=1}^c f_1(p_i)\leq \sum_{i=1}^c e(X_i,S)=e(S,K(7,3)\setminus S)\leq \min\left(4|S|,\left\lfloor\dfrac{7|S|(35-|S|)}{35}\right\rfloor\right).\end{equation}
    
    Let $a$ and $b$ be the number of components of size $1$ and $2$ respectively. The size of the neighborhood of the $a$ singletons and $b$ $K_2$'s must be at most $|S|$. A computer search over all possible values of $c$ and $|S|$ confirms that no partitions satisfy these two conditions among other trivial conditions on $a,b$. See Section \ref{Appendix:p} for more details.
    \end{proof}
    
    \section{Toughness of \texorpdfstring{$K(n,4)$}{K(n,4)}}\label{sec5}
    
    \begin{thm}
    If $n\geq13$, then $t(K(n,4))=\frac{n}{4}-1$. Moreover, if $S$ is a vertex cut of $K(n,4)$ such that $\frac{|S|}{c(K(n,4)\setminus S)} = \frac{n}{4}-1$,  then $S$ is the complement of a maximum independent set of $K(n,4)$.
    \end{thm}
    
    \begin{proof}
    Suppose $S$ is a vertex cut of $K(n,4)$ that is not the complement of a maximum independent set such that $\frac{|S|}{c(K(n,4)\setminus S)}\le\frac{n}{4}-1$.
    Let $c=c(K(n,4)\setminus S)$. By Theorem \ref{TK}, we have $\binom{n-5}{3}+1\leq c\leq \binom{n-1}{3}-\binom{n-5}{3}+1$. Also,
    \begin{align}
    |S|&\le \dfrac{n-4}{4}c,\label{14,4:eq2}\\
    |K(n,4)\setminus S|&\ge\binom{n}{4} - \dfrac{n-4}{4}c.\label{14,4:eq3}
    \end{align}
    We would like $|K(n,4)\setminus S|\geq 2c$ in order to apply Lemma \ref{minlemma}. By the upper bound on $c$, this is satisfied when
    \begin{equation*}
    \binom{n}{4}\geq \dfrac{n+4}{4}\left(\binom{n-1}{3}-\binom{n-5}{3}+1\right).
    \end{equation*}
    The above equation holds if and only if $n\geq 13$.
    Thus we may apply Lemma \ref{minlemma} to create a vertex partition of $K(n,4)\setminus S$ into two vertex sets $S_1,S_2$ with no edges between $S_1,S_2$ and $\min(|S_1|,|S_2|)\geq c-1$. Now by \eqref{HMP2},
    \[
    |S|\geq \dfrac{6(n-8)(n-5)}{(n-3)(n-2)^3}|S_1||S_2|\geq \dfrac{6(n-8)(n-5)}{(n-3)(n-2)^3}(c-1)\left(|K(n,4)\setminus S|-(c-1)\right).
    \]
    Using the bounds on $c$ and \eqref{14,4:eq3}, we can rewrite the above lower bound on $|S|$ as
    \begin{align*}
        |S|&\geq \dfrac{6(n-8)(n-5)}{(n-3)(n-2)^3}(c-1)\left(\binom{n}{4}-\dfrac{n}{4}c+1\right)\\
        &= \dfrac{6(n-8)(n-5)}{(n-3)(n-2)^3}\dfrac{n}{4}(c-1)\left(\binom{n-1}{3}+\dfrac{4}{n}-1-(c-1)\right).
    \end{align*}
    Notice $c-1\in [\binom{n-5}{3},\binom{n-1}{3}-\binom{n-5}{3}]$ and $\binom{n-1}{3}+\frac{4}{n}-1 < \binom{n-1}{3}$. Then the lower bound of $|S|$ above is a quadratic in $c-1$ that is minimized when $c-1=\binom{n-1}{3}-\binom{n-5}{3}$. Thus,
    \[|S|\geq \dfrac{6(n-8)(n-5)}{(n-3)(n-2)^3}\dfrac{n}{4}\left(\binom{n-1}{3}-\binom{n-5}{3}\right)\left(\dfrac{4}{n}-1+\binom{n-5}{3}\right).\]
    Now from the upper bound on $c$ and \eqref{14,4:eq2}, we have the upper bound \[|S|\leq \frac{n-4}{4}\left(\binom{n-1}{3}-\binom{n-5}{3}+1\right).\]
    For $n\geq 13$, the lower bound of $|S|$ is strictly greater than the upper bound of $|S|$, yielding a contradiction.
    \end{proof}
    
    
    
    For $K(12,4)$, we will require Lemmas \ref{KKBounds} and \ref{switch2}.
    
    \begin{thm}
    We have $t(K(12,4))=2$. Moreover, if $S$ is a vertex cut of $K(12,4)$ such that $\frac{|S|}{c(K(12,4)\setminus S)} = 2$,  then $S$ is the complement of a maximum independent set of $K(12,4)$.
    \end{thm}
    \begin{proof}
    Suppose $S$ is a vertex cut of $K(12,4)$ that is not the complement of a maximum independent set such that $\frac{|S|}{c(K(12,4)\setminus S)}\le 2$.
    Let $c=c(K(12,4)\setminus S)$. By Theorem \ref{TK}, $36\leq c\leq 131$. Then $|K(12,4)\setminus S|\geq 495 - 2c$. For $36 \leq c \leq 123$, we have $495 -2c \geq 2c$, from which Lemma \ref{minlemma} gives a vertex partition $S_1,S_2$ with at least $c-1$ vertices in each subset. Therefore
    $$|S| \geq \frac{7}{375} |S_1||S_2| \geq \frac{7}{375} (c-1)\left( 495 -2c - (c-1) \right) = \frac{7}{375} (c-1)(496-3c). $$
    For $36 \leq c \leq 129$, this lower bound contradicts $|S|\leq 2c$. It is left to check when $c=130,131$. By Lemma \ref{switch2}, at most one component of $K(12,4)\setminus S$ is not a singleton. There are at least $129$ singletons in $K(12,4)\setminus S$, whose neighborhood must lie in $S$. By Lemma \ref{KKBounds}, the neighborhood contains at least $329$ vertices, contradicting $|S|\leq 2c\leq 262$.
    \end{proof}
    
    For $K(11,4)$, we will need a vertex of high degree to form an independent set large enough to use Lemma \ref{switch2}.
    
    \begin{thm}
    We have $t(K(11,4))=\frac{7}{4}$. Moreover, if $S$ is a vertex cut of $K(11,4)$ such that $\frac{|S|}{c(K(11,4)\setminus S)} = \frac{7}{4}$,  then $S$ is the complement of a maximum independent set of $K(11,4)$.
    \end{thm}
    \begin{proof}
    Suppose $S$ is a vertex cut of $K(11,4)$ that is not the complement of a maximum independent set such that $\frac{|S|}{c(K(11,4)\setminus S)}\leq \frac{7}{4}$.
    Let $c=c(K(11,4)\setminus S)$. By Theorem \ref{TK}, $21\leq c\leq 101$. Then $|K(11,4)\setminus S|\geq 330-\frac{7}{4}c$.
    For $21\leq c\leq 85$, we have $330-\frac{7}{4}c\geq 2c$, from which Lemma \ref{minlemma} gives a vertex partition $S_1,S_2$ with at least $c-1$ vertices in each subset. Therefore
    \[|S| \geq \frac{1}{54} |S_1||S_2| \geq \frac{1}{54} ( c-1 )\left( 330 - \frac74 c - (c-1) \right) = \frac{1}{54} (c-1)\left(331-\frac{11}{4} c\right). \]
    For $21\leq c\leq 85$, this lower bound of $|S|$ contradicts  $|S|\leq\frac{7}{4}c$. It is left to check $86\leq c\leq 101$. Let $X_1,\ldots, X_c$ denote the components of $K(11,4)\setminus S$. Let $\Delta(X_i)$ denote the maximum degree of the induced subgraph of component $X_i$ and $d=\max_{i\in[c]}\Delta(X_i)$. Then each vertex in $K(11,4)\setminus S$  contributes at least $35-d$ edges into $S$. Thus,
    \[e(S,K(11,4)\setminus S)\geq (35-d)|K(11,4)\setminus S|.\]
    From Lemma \ref{laplacian}, we also have the upper bound
    \[e(S,K(11,4)\setminus S)\leq \dfrac{55|S|(330-|S|)}{330}.\]
    Combining the two bounds on $e(S,K(11,4)\setminus S)$ and $|S|\leq \frac{7}{4}c$ gives
    \begin{equation}
    d\geq 35-\dfrac{|S|}{6}\geq35-\frac{7}{24}c.\label{11,4:dbound}
    \end{equation}
    
    Let $v$ be a vertex of degree $d$ in component $\C$. Since $K(11,4)$ is triangle-free, the neighborhood $N_{\C}(v)$ within component $\C$ forms an independent set of vertices. Let $H$ be an independent set formed by taking a vertex from each component except $\C$. Then $N_{\C}(v)\cup H$ is an intersecting family of $d+c-1$ vertices. By \eqref{11,4:dbound} and $c\geq 86$,
    \[|N_{\C}(v)\cup H|=d+c-1\geq 34+\frac{17}{24}c\geq 94,\]
    which exceeds the bound given by Lemma \ref{switch2}, so by the moreover part of the lemma, there cannot exist two components other than $\C$ that are not singletons. Then at most one component other than $\C$ is not a singleton. Thus $K(11,4)\setminus S$ contains at least $c-2\geq 84$ singletons whose neighborhood must be contained in $S$. By Lemma \ref{KKBounds}, there are at least $203$ vertices in the neighborhood, a contradiction with $|S|\leq \frac{7}{4}c< 177$.
    \end{proof}
    
    For $K(10,4)$, we use a similar edge counting argument from $K(7,3)$.
    
    \begin{thm}
    We have $t(K(10,4))=\frac{3}{2}$. Moreover, if $S$ is a vertex cut of $K(10,4)$ such that $\frac{|S|}{c(K(10,4)\setminus S)} = \frac{3}{2}$,  then $S$ is the complement of a maximum independent set of $K(10,4)$.
    \end{thm}
    
    \begin{proof}
    Suppose $S$ is a vertex cut of $K(10,4)$ achieving toughness such that $S$ is not the complement of a maximum independent set and $\frac{|S|}{c(K(10,4)\setminus S)}\leq \frac{3}{2}$.
    Let $c=c(K(10,4)\setminus S)$. By Theorem \ref{TK}, $11\leq c\leq 75$. It follows that, $|K(10,4)\setminus S|\geq 210 - \frac{3}{2}c$.
    For $11\leq c\leq 47$, we have $210-\frac{3}{2}c\geq 2c$, from which Lemma \ref{minlemma} gives a vertex partition $S_1,S_2$ with at least $c-1$ vertices in each subset. Therefore
    \[|S|\geq \dfrac{15}{896}|S_1||S_2|\geq \dfrac{15}{896}(c-1)\left(210-\dfrac{3}{2}c-(c-1)\right)=\dfrac{15}{896}(c-1)\left(211-\dfrac{5}{2}c\right).\]
    For $11\leq c\leq 47$, this lower bound of $|S|$ contradicts  $|S|\leq\frac{3}{2}c$. It is left to check $48\leq c\leq 75$. Since $K(10,4)$ is $15$-regular, we have that $e(S,K(10,4)\setminus S)\leq 15|S|$. By Lemma \ref{laplacian}, we may improve this upper bound to
    \begin{equation}\label{10,4:ehigh}
        e(S,K(10,4)\setminus S)\leq \min\left(15|S|,\left\lfloor\dfrac{25|S|(210-|S|)}{210}\right\rfloor\right).
    \end{equation}
    Now for any component $\C$ of $K(10,4)\setminus S$, we will bound $e(\C,S)$ from below. By Lemma \ref{laplacian}, we have
    \begin{equation*}\label{10,4:xhigh}
        e(\C,S)\geq \dfrac{9|\C|(210-|\C|)}{210}.
    \end{equation*}
    We will require tighter bounds for small $|\C|$. When $|\C|=1,2$, we obtain $e(\C,S)=15,28$. Since $K(10,4)$ has girth 4 \cite{girth_ref}, if $|\C|=3$, then $\C$ is a tree violating Lemma \ref{trees}. For $|\C|=4,5$, the maximum of $e(\C)$ will be 4 and 6, respectively. Since $K(10,4)$ is $15$-regular,
    \[e(\C,S)=15|\C|-2e(\C).\]
    Thus $e(\C,S)\geq 52, 63$ for $|\C|=4,5$. We combine this into a single function $f_2(|\C|)$ which provides a lower bound for $e(\C,S)$. Define
    \begin{equation}\label{10,4:f}
    f_2(t)=\begin{cases}
    15 & \text{if }t=1\\
    28 & \text{if }t=2\\
    52 & \text{if }t=4\\
    63 & \text{if }t=5\\
    \ceil{\frac{9t(210-t)}{210}} & \text{if }t\ge 6
    \end{cases}.\end{equation}
    Let $X_1,\ldots,X_c$ denote the components of $K(10,4)\setminus S$. Let $p_i=|X_i|$, and without loss of generality assume $p_1\leq \cdots \leq p_c$. By \eqref{10,4:ehigh} and \eqref{10,4:f}, $(p_1,\ldots,p_c)$ is a partition in $\mathcal{P}(210-|S|,c,\{3\})$ satisfying
    \begin{equation}\label{10,4:check}\sum_{i=1}^c f_2(p_i)\leq \sum_{i=1}^c e(X_i,S) = e(S,K(10,4)\setminus S)\leq \min\left(15|S|,\left\lfloor\dfrac{25|S|(210-|S|)}{210}\right\rfloor\right).\end{equation}
    Let $a$ and $b$ be the number of components of size $1$ and $2$ respectively. Again, the size of the neighborhood of the $a$ singletons and $b$ $K_2$'s must be at most $|S|$. A computer search in Section \ref{Appendix:p} over all possible values of $c$ and $|S|$ confirms that no partitions satisfy these two conditions among other trivial conditions on $a,b$.
    \end{proof}
    
    
    \begin{thm}
    We have $t(K(9,4))=\frac{5}{4}$. Moreover, if $S$ is a vertex cut of $K(9,4)$ such that $\frac{|S|}{c(K(9,4)\setminus S)} = \frac{5}{4}$,  then $S$ is the complement of a maximum independent set of $K(9,4)$.
    \end{thm}
    
    \begin{proof}
    Suppose $S$ is a vertex cut of $K(9,4)$ achieving toughness such that $S$ is not the complement of a maximum independent set and $\frac{|S|}{c(K(9,4)\setminus S)}\leq\frac{5}{4}$.
    Let $c=c(K(9,4)\setminus S)$. By Theorem \ref{TK}, $5\leq c\leq 53$. By Lemma \ref{laplacian} and the regularity of $K(9,4)$,
    \begin{equation}\label{9,4:ehigh}
        e(S,K(9,4)\setminus S)\leq \min\left(5|S|,\left\lfloor\dfrac{9|S|(126-|S|)}{126}\right\rfloor\right).
    \end{equation}
    Now for any component $\C$ of $K(10,4)\setminus S$, Lemma \ref{laplacian} implies
    \[
        e(\C,S)\geq \dfrac{2|\C|(126-|\C|)}{126}.
    \]
    We will require tighter bounds on $e(\C,S)$ for small $|\C|$. By Lemma \ref{trees}, $\C$ may not be a tree. Since $K(9,4)$ has girth $6$ \cite{girth_ref}, we do not consider when $3\leq |\C|\leq 5$. For $6\leq|\C|\leq 16$, the bounds on $e(\C,S)$ were found by the computer search in Section \ref{Appendix:e} on all possible components of order $|\C|$. Define
    \begin{equation}\label{9,4:f}
    f_3(t) = \begin{cases}
    5 & \text{if }t=1\\
    8 & \text{if }t=2\\
    18 & \text{if }t=6\\
    21 & \text{if }t=7\\
    22 & \text{if }t=8\\
    25 & \text{if }t=9\\
    26 & \text{if }t=10\\
    27 & \text{if }t=11\\
    28 & \text{if }t=12\\
    29 & \text{if }t=13\\
    28 & \text{if }t=14\\
    31 & \text{if }t=15\\
    32 & \text{if }t=16\\
    \ceil{\frac{2t(126-t)}{126}} & \text{if } t\ge 17
    \end{cases}.\end{equation}
    
    Let $X_1,\ldots,X_c$ denote the components of $K(9,4) \setminus S$. Let $p_i=|X_i|$ and without loss of generality assume $p_1\leq \cdots \leq p_c$. By \eqref{9,4:ehigh} and \eqref{9,4:f}, $(p_1,\ldots,p_c)$ is a partition in $\mathcal{P}(126-|S|,c,\{3,4,5\})$ satisfying
    \begin{equation}\label{9,4:check}\sum_{i=1}^c f_3(p_i)\leq \sum_{i=1}^c e(X_i,S)=e(S,K(9,4)\setminus S)\leq \min\left(5|S|,\left\lfloor\dfrac{9|S|(126-|S|)}{126}\right\rfloor\right).\end{equation}
    Let $a$ and $b$ be the number of components of size $1$ and $2$ respectively. Again, the neighborhood of the $a$ singletons and $b$ $K_2$'s must be at most $|S|$. A computer search in Section \ref{Appendix:p} over all possible values of $c$ and $|S|$ confirms that no partitions satisfy these two conditions among other trivial conditions on $a,b$.
    \end{proof}
    
    \section{Final Remarks}
    
    In this paper, we show that the toughness of the Kneser graph $K(n,k)$ equals $n/k-1$ for $k\in \{3,4\}$ and any $n\geq 2k+1$ and for given $k\geq 5$ and sufficiently large $n$ (as a function of $k$). We conjecture that this result holds for any $k\geq 5$ and $n\geq 2k+1$. It would be interesting to determine the toughness of other families of regular graphs such as the Johnson graphs, Paley graphs, block graphs of Steiner systems, and $q$-analogs of Kneser graphs.
    

    \section{Appendix}
    See \url{https://github.com/aostuni/kneser-toughness} for source files, including program outputs.
    
    \subsection{Calculation of lower bounds of \texorpdfstring{$e(\C,S)$}{e(C,S)}.}\label{Appendix:e}
    For $6\leq |\C|\leq 16$, we rely on lower bounds of $e(\C,S)$ where $\C$ is a component of $K(9,4)$ to create the function $f_3$ in \eqref{9,4:f}. Here, we describe in detail how these values were determined.
    
    With geng, we generate biconnected graphs with maximum degree $5$ and girth at least $5$ (the geng tool does not allow generation of graphs of girth above 5). With countg, we remove graphs of girth $5$ and find an upper bound on $e(\C)$.
    \begin{verbatim}
    $ geng -CtfD5 p | countg -g6: --e
    \end{verbatim}
    The total execution time took at most $15$ seconds on our machine. Upper bounds on $e(\C)$ are shown below.
    \[
    \begin{array}{|c|c|}
    \hline
    |\C| & e(\C)\\\hline
    6 & 6\\\hline
    7 & 7\\\hline
    8 & 9\\\hline
    9 & 10\\\hline
    10 & 12\\\hline
    11 & 14\\\hline
    12 & 16\\\hline
    13 & 18\\\hline
    14 & 21\\\hline
    15 & 22\\\hline
    16 & 24\\\hline
    \end{array}
    \]
    Lower bounds on $e(\C,S)$ can be calculated using $e(\C,S)=5|\C|-2e(\C)$.
    
    \subsection{Partition searching algorithm}\label{Appendix:p}
    We describe the algorithm used to check the partitions of vertices into the $c$ components for a vertex cut $S$ of each of $K(7,3)$, $K(10,4)$, and $K(9,4)$. Denote $g$ and $\ell$ to be the girth and degree of $K(n,k)$ respectively. Recall the functions $f_1$, $f_2$, and $f_3$ defined for each $K(7,3)$, $K(10,4)$, and $K(9,4)$, respectively. Let $f$ denote this function for general $n,k$. In our implementation, we store values of $f$ in an array up to a max component size of
    \[|K(n,k)\setminus S|-(c-1)\leq \binom{n}{k}-|S|-c+1\leq \binom{n}{k}-2c.\] The last inequality follows by Lemma \ref{toughone}. Let $L$ and $U$ be the lower and upper bounds of the inequality on $e(S, K(n,k)\setminus S)$, namely \eqref{7,3:check}, \eqref{10,4:check}, and \eqref{9,4:check}. For general $n,k$, define
    \begin{align}
        &L((p_1\ldots,p_c))=\sum_{i=1}^c f(p_i),\\
        &U(|S|)=\min\left(\binom{n-k}{k}|S|,\left\lfloor\dfrac{\left(\binom{n-k}{k}+\binom{n-k-1}{k-1}\right)|S|\left(\binom{n}{k}-|S|\right)}{\binom{n}{k}}\right\rfloor\right).
    \end{align}
    Let $a$ and $b$ be the number of singletons and $K_2$'s respectively in $K(n,k)\setminus S$. The next largest component size must be $g$ by Lemma \ref{trees}. By counting the components and vertices in $K(n,k)\setminus S$, we immediately have
    \begin{align}
        a+b&\leq c,\label{Appendix:c}\\
        a+2b+g(c-a-b)&\leq |K(n,k)\setminus S|=\binom{n}{k}-|S|.\label{Appendix:v}
    \end{align}
    Now count the edges coming out of each component. Trivially $f(1)=\ell$ and $f(2)=2\ell-2$. By inspecting the values of $f$ for each of $K(7,3)$, $K(10,4)$, and $K(9,4)$, we find that $f(g)\leq f(t)$ for all $g\leq t\leq \binom{n}{k}-2c$. Thus,
    \begin{equation}\label{Appendix:e1}
        \ell a+(2\ell-2)b+f(g)(c-a-b) \leq e(S,K(n,k)\setminus S)\leq U(|S|).
    \end{equation}
    More generally, for a partition $p=(p_1,\ldots,p_c)$ of the $K(n,k)\setminus S$ vertices into $c$ components,
    \begin{equation}\label{Appendix:e2}
        L(p) \leq e(S,K(n,k)\setminus S)\leq U(|S|).
    \end{equation}
    Note \eqref{Appendix:e2} is exactly \eqref{7,3:check}, \eqref{10,4:check}, and \eqref{9,4:check}. Finally, we count the size of the induced neighborhood of the singletons and $K_2$'s in $S$. First, the neighborhood of any $r$ of the $a$ singletons lie in $S$. Define $KK(|Q|)$ to be the lower bound of $N(Q)$ in Lemma \ref{KKBounds}. Then for all $1\leq r\leq a$,
    \begin{equation}\label{Appendix:n1}
        KK(r)\leq |S|.
    \end{equation}
    Let $H$ be the set of $a$ singletons. For $1\leq r\leq b$, choose a vertex from $r$ of the $b$ $K_2$'s, and add these to $H$. Now the neighborhood of $H$ consists of $r$ vertices in $K(n,k)\setminus S$ and vertices in $S$. For all $1\leq r\leq b$,
    \begin{equation}\label{Appendix:n2}
        KK(a+r)-r\leq |S|.
    \end{equation}
    Finally, we present our algorithm for finding all partitions for some $c$ which satisfy all the conditions above. We iterate this algorithm over all possible values of $c$ from the spectral bounds.
    
    \begin{algorithm}[H]
    \caption{Partition algorithm for $K(n,k)$}
    \begin{algorithmic}[1]
    \Procedure{PartitionSearch}{$c$}
    \For{$|S|=c+1$ \To $\floor{\left(\frac{n}{k}-1\right)c}$}
        \ForAll{$a,b\geq 0$ \textbf{satisfying} \eqref{Appendix:c}, \eqref{Appendix:v}, \eqref{Appendix:e1}, \eqref{Appendix:n1}, \eqref{Appendix:n2}}
            \ForAll{$p\in \mathcal{P}(\binom{n}{k}-|S|-a-2b,c-a-b,\{1,\ldots,g-1\})$}
                \If{$L(p)+da+(2d-2)b\leq U(|S|)$}
                    \State \textbf{print }$(1^a,2^b)\cup p$
                \EndIf
            \EndFor
        \EndFor
    \EndFor
    \EndProcedure
    \end{algorithmic}
    \end{algorithm}
    
    Note the check in line $5$ is exactly condition \eqref{Appendix:e2}. The generation of partitions with minimum part $g$ can be easily implemented by a slight modification to Knuth's co-lexicographic partition generation algorithm \cite{Knuth}. The algorithm took less than a second to run on our machine for each of $K(7,3)$, $K(10,4)$, and $K(9,4)$. No partitions were found.
    
    \section{Acknowledgements}
    The authors thank the anonymous reviewers, whose thorough comments and feedback have greatly improved the paper. Davin Park was supported by the Combinatorics and Algorithms for Real Problems REU at the University of Maryland, College Park, under NSF grant CNS-1560193. Sebastian Cioab\u{a} is partially supported by NSF grants DMS-1600768 and CIF-1815922. Tanay Wakhare is supported by an MIT Television and Signal Processing Fellowship.
    
    
    \bibliographystyle{plain}
    \bibliography{bibliography}
    \end{document}